\documentclass[11pt]{article}

\usepackage[letterpaper]{geometry}
\usepackage{amsmath}
\usepackage{amsthm}
\usepackage{newtxtext}
\usepackage{newtxmath}
\usepackage[colorlinks=true, allcolors=blue]{hyperref}
\usepackage{enumitem}
\usepackage{microtype}

\setlist[itemize]{noitemsep}
\setlist[enumerate]{noitemsep}

\newtheorem{theorem}{Theorem}

\theoremstyle{definition}

\newcommand{\fn}[1]{\mathsf{#1}}
\newcommand{\ex}[1]{\exists #1. \,}
\newcommand{\fa}[1]{\forall #1. \,}
\newcommand{\liff}{\leftrightarrow}
\newcommand{\limplies}{\to}

\begin{document}

\title{An Impossible Asylum}
\markright{An Impossible Asylum}
\author{Jeremy Avigad, Seulkee Baek, Alexander Bentkamp,\\ Marijn Heule, and Wojciech Nawrocki}

\maketitle

\begin{abstract}
In 1982, Raymond Smullyan published an article, ``The Asylum of Doctor Tarr and Professor Fether,'' that consists of a series of puzzles. These were later reprinted in the anthology \emph{The Lady or The Tiger? and Other Logic Puzzles}. The last puzzle, which describes the asylum alluded to in the title, was designed to be especially difficult. With the help of automated reasoning, we show that the puzzle's hypotheses are, in fact, inconsistent, which is to say, no such asylum can possibly exist.
\end{abstract}

\section{Introduction}

The logician, author, and sleight-of-hand magician Raymond Smullyan was one of the most celebrated puzzle makers of all time. In addition to a number of philosophical works and logic textbooks, he published more than a dozen collections of puzzles, including \emph{Alice in Puzzle-Land} and \emph{King Arthur in Search of His Dog}. His book, \emph{G\"odel's Incompleteness Theorems}, is an accessible introduction to Kurt G\"odel's profound results on the limits of formal provability.

In 1982, in \emph{The Two-Year College Mathematics Journal}, he presented a series of puzzles involving a fictional Inspector Craig of Scotland Yard and some asylums in France \cite{smullyan}. The series was later reprinted in his anthology, \emph{The Lady or the Tiger? and Other Logic Puzzles}. In the asylums, every inhabitant is either a doctor or a patient,
and every inhabitant is either sane or insane. Moreover, the sane inhabitants are entirely sane and the insane inhabitants are entirely insane in the following sense: for any proposition $P$, a sane inhabitant believes $P$ if and only if $P$ is true and an insane inhabitant believes $P$ if and only if $P$ is false.

In the next section, we reproduce the last of these puzzles. Smullyan's solutions show that the hypotheses lead to a shocking conclusion: in this particular asylum, all the patients are sane and all the doctors are insane. The puzzle is thus an homage to a story by Edgar Allen Poe, ``The System of Doctor Tarr and Professor Fether,'' which features an asylum in which the patients have imprisoned the doctors and assumed their roles.

In this article, we show something stronger, namely, that the hypotheses themselves are inconsistent. Whether this conclusion is more or less shocking than Smullyan's, it does serve to mitigate the horror: there cannot be any asylum meeting his description.

In a well-known essay, ``The ways of paradox'' \cite{quine}, the philosopher W.~V.~O.~Quine identified a class of paradoxes that have prima facie plausible assumptions but cannot be realized. For example, starting with the assumption that the male barber of a certain town shaves all and only those men that do not shave themselves, we reach the absurd conclusion that the barber both does and does not shave himself. We must therefore conclude that there is no such barber. The same is true of the asylum in the puzzle: the fact that the hypotheses lead to a contradiction tells us simply that there there is not, and cannot be, such an asylum.

This raises interesting philosophical questions. Fictional narratives require us to suspend disbelief and often present us with counterfactual or implausible scenarios. Such narratives can violate physical laws, as do stories of magic or time travel. Douglas Hofstadter has even challenged us to imagine worlds that violate mathematical laws, for example, a world where $\pi$ is equal to 3 \cite{hofstadter}. But logical impossibility is an especially hard pill to swallow. We often accord fictional characters a provisional existence; for example, Sherlock Holmes is a detective in Arthur Conan Doyle's fictional depiction of London at the turn of the twentieth century. But does it make any sense at all to talk about the asylum of Doctor Tarr and Professor Fether in Smullyan's narrative, knowing that such an asylum cannot possibly exist?

\section{The puzzle}
\label{section:the:puzzle}

The last asylum puzzle, as presented in \emph{The Lady or the Tiger}, runs as follows.

\begin{quotation}
\noindent The last asylum Craig visited he found to be the most bizarre of all. This asylum was run by two doctors named Doctor Tarr and Professor Fether. There were other doctors on the staff as well. Now, an inhabitant was called \emph{peculiar} if he believed that he was a patient. An inhabitant was called \emph{special} if all patients believed he was peculiar and no doctor believed he was peculiar. Inspector Craig found out that at least one inhabitant was sane and that the following condition held:

\emph{Condition C:} Each inhabitant had a best friend in the asylum. Moreover, given any two inhabitants, A and B, if A believed that B was special, then A's best friend believed that B was a patient.

Shortly after this discovery, Inspector Craig had private interviews with Doctor Tarr and Professor Fether. Here is the interview with Doctor Tarr:

\emph{Craig:} Tell me, Doctor Tarr, are all the doctors in this asylum sane?

\emph{Tarr:} Of course they are!

\emph{Craig:} What about the patients? Are they all insane?

\emph{Tarr:} At least one of them is.

The second answer struck Craig as a surprisingly modest claim! Of course, if all the patients are insane, then it certainly is true that at least one is. But why was Doctor Tarr being so cautious? Craig then had his interview with Professor Fether, which went as follows:

\emph{Craig:} Doctor Tarr said that at least one patient here is insane. Surely that is true, isn't it?

\emph{Professor Fether:} Of course it is true! All the patients in this asylum are insane! What kind of asylum do you think we are running?

\emph{Craig:} What about the doctors? Are they all sane?

\emph{Professor Fether:} At least one of them is.

\emph{Craig:} What about Doctor Tarr? Is he sane?

\emph{Professor Fether:} Of course he is! How dare you ask me such a question?

At this point, Craig realized the full horror of the situation! What was it?
\end{quotation}

According to Smullyan, we can assume that when an inhabitant of an asylum says something, they believe it. So each of the following claims is either implicit or explicit in this story.

\begin{enumerate}
\item Tarr is a doctor.
\item Fether is a doctor.
\item There are other doctors in the asylum.
\item By definition, an inhabitant A is \emph{peculiar} if A believes that A is a patient.
\item By definition, an inhabitant A is \emph{special} if all patients believe that A is peculiar and no doctor believes that A is peculiar.
\item At least one inhabitant is sane.
\item Condition C: Given any two inhabitants, A and B, if A believes that B is special, then A's best friend believes that B is a patient.
\item Tarr believes that every doctor is sane.
\item Tarr believes that at least one patient is insane.
\item Fether believes that every patient is insane.
\item Fether believes that at least one doctor is sane.
\item Fether believes that Tarr is sane.
\end{enumerate}

In the list that follows, we express these hypotheses in the language of symbolic logic. The expression $\fn{Doctor}(\fn{Tarr})$ says that Tarr is a doctor. Here, the symbol $\fn{Doctor}(\cdot)$ represents a \emph{predicate}, otherwise known as a \emph{unary relation}. An expression $\ex x \ldots$ should be read ``there exists an $x$ such that \ldots'' and an expression $\fa x \ldots$ should be read ``for every $x$ \ldots\,.'' The symbols $\land$, $\limplies$, $\liff$, $\lnot$ stand for ``and,'' ``implies,'' ``if and only if,'' and ``not,'' respectively.

In our formal representation of the hypotheses, we express the fact that someone is a patient by saying that they are not a doctor, and we express the fact that someone is insane by saying that they are not sane. We repeatedly make use of the fact that, according to the rules that Smullyan has laid out, a statement of the form ``$x$ believes $P$'' is true if and only if $x$ is sane and $P$ holds or $x$ is insane and $P$ doesn't hold. This is logically equivalent to saying ``$x$ is sane if and only if $P$,'' which is represented by the expression $\fn{Sane}(x) \liff P$.

\begin{enumerate}
\item $\fn{Doctor}(\fn{Tarr})$
\item $\fn{Doctor}(\fn{Fether})$
\item $\ex x \fn{Doctor}(x) \land x \ne \fn{Tarr} \land x \ne \fn{Fether}$
\item $\fa x \fn{Peculiar}(x) \liff (\fn{Sane}(x) \liff \lnot \fn{Doctor}(x))$
\item $\fa x \fn{Special}(x) \liff (\fa y \lnot \fn{Doctor}(y) \liff (\fn{Sane}(y) \liff \fn{Peculiar}(x)))$
\item $\ex x \fn{Sane}(x)$
\item $\fa x \fa y (\fn{Sane}(x) \liff \fn{Special}(y)) \limplies (\fn{Sane}(\fn{bestFriend}(x)) \liff \lnot \fn{Doctor}(y))$
\item $\fn{Sane}(\fn{Tarr}) \liff \fa x \fn{Doctor}(x) \limplies \fn{Sane}(x)$
\item $\fn{Sane}(\fn{Tarr}) \liff \ex x \lnot \fn{Doctor}(x) \land \lnot \fn{Sane}(x)$
\item $\fn{Sane}(\fn{Fether}) \liff \fa x \lnot \fn{Doctor}(x) \limplies \lnot \fn{Sane}(x)$
\item $\fn{Sane}(\fn{Fether}) \liff \ex x \fn{Doctor}(x) \land \fn{Sane}(x)$
\item $\fn{Sane}(\fn{Fether}) \liff \fn{Sane}(\fn{Tarr})$.
\end{enumerate}
Notice that the definitions in 4 and 5 are presented as ``if and only if'' statements. We express condition $C$ by choosing a function symbol $\fn{bestFriend}(x)$ to denote $x$'s best friend.

The statements above all belong to the system of logic known as \emph{first-order logic}. The words ``first order'' indicate that the quantifiers range over objects in the intended interpretation, in this case, individuals in the intended asylum. There is no quantification over sets of individuals, relationships between individuals, or functions.

Notice that we could equally well have introduced new predicates $\fn{Insane}(x)$ and $\fn{Patient}(x)$, together with additional axioms $\fa x \fn{Insane}(x) \liff \lnot \fn{Sane}(x)$ and $\fa x \fn{Patient}(x) \liff \lnot \fn{Doctor}(x)$. This would result in what logicians call a \emph{definitional extension} of our original formulation, and it would not affect the conclusions that can be derived in the original language. This illustrates the fact that there are often choices to be made when rendering informal statements in a formal language.

Modern logic provides us with the definition of a \emph{model} for a first-order language, which consists of a set of objects together with interpretations of the predicate and function symbols as predicates and functions on those objects. It also provides a definition of what it means for a sentence to be \emph{true} in a model. Finally, logic provides us with notions of formal \emph{derivation}, or \emph{proof}.
Conventional proof systems for first-order logic are \emph{sound}, which means that anything provable from a set of hypotheses is true in every model of those hypotheses. They are also \emph{complete}, which means that if something is true in every model of some hypotheses, then it is formally provable from those hypotheses. Soundness and completeness serve to connect formal notions of proof to the formal notion of what it means to be true in a model. This relationship is fundamental to mathematical logic. Another result that is central to mathematical logic is the fact that first-order logic is \emph{undecidable}. This means that although we can write a computer program to search systematically for a proof of a given first-order statement, there is no algorithm that can determine, in general, whether such a proof exists. In other words, when there is no such proof, we have to reconcile ourselves to the fact that the search may run forever, and we may never know when it is time to give up. For more on mathematical logic, see, for example, \cite{avigad:book,enderton,rautenberg}.

Below we will show that it is possible to derive a contradiction from Smullyan's assumptions. By soundness and completeness, this is equivalent to saying that there is no model that satisfies all the constraints. In fact, we will see that given the definition of \emph{peculiar}, the definition of \emph{special}, and condition C, only three additional hypotheses are needed to derive a contradiction.

\section{Formal methods and \emph{Vampire}}

In the fall of 2021, three of us taught an undergraduate course
at Carnegie Mellon University
called \emph{Logic and Mechanized Reasoning}, an introduction to logic and formal methods for students in computer science \cite{lamr}. \emph{Formal methods} refers to a body of logical methods and software tools that are used to verify the correctness of hardware and software systems. Such tools include \emph{SAT solvers}, which are used to solve large sets of constraints that can be encoded as formulas in propositional logic or show that no such solution exists; \emph{SMT solvers}, which do the same for a richer language that includes primitives for real and integer arithmetic; \emph{automated theorem provers}, which carry out searches for formal derivations; and \emph{interactive theorem provers}, which help users construct complex formal derivations by hand. For more on automated reasoning, see \cite{handbook:satisfiability,harrison,kroening:strichman,handbook:automated:reasoning}, and for applications to mathematics, see \cite{avigad,heule:kullmann}.

A novel feature of the course was that students had the opportunity to experiment with these tools and use them to solve interesting problems and puzzles. Toward the end of the semester, we showed them how to express the hypotheses of one of Smullyan's logic puzzles in first-order logic and use an automated theorem prover called \emph{Vampire} to derive the expected conclusion \cite{kovacs:voronkov}. For the last homework assignment, we asked them to use Vampire to show, following Smullyan, that the hypotheses of the last asylum puzzle imply that all the patients are sane and all the doctors are insane.

When we tried this ourselves, however, we encountered an unexpected result. When writing statements in the language of formal logic, it is easy to make mistakes, such as switching $x$ and $y$ in a part of a formula or leaving out needed parentheses. The output of a formal tool is only as good as the input, and it is not always immediately clear whether the input means what we think it means. Both mathematicians and computer scientists are familiar with the practice of carrying out \emph{sanity checks}: before trying to prove a hard theorem, we instantiate the hypotheses to a familiar structure and convince ourselves that the conclusion makes sense; before trying to prove a complex identity, we compute the left- and right-hand sides at a few values to make sure they agree; and before deploying an algorithm, we test it on a few inputs to make sure it behaves as intended. When it comes to formal theorem proving, there is a natural sanity check, namely, checking that the hypotheses of a putative theorem are \emph{consistent}. In other words, we check that they do not imply a contradiction. Knowing that the hypotheses of a formal theorem are consistent is not enough to guarantee that the theorem has the intended meaning, but when the hypotheses are \emph{inconsistent}, it is almost always a sign that we have done something wrong.

When we asked Vampire to prove a contradiction from the hypotheses, it succeeded in doing so within a few dozen milliseconds. With a little trial and error, we were able to reduce the set of contradictory hypotheses to 4, 5, 7, 8, 10, and 12. Vampire expects its input in \emph{TPTP} format, where ``TPTP'' stands for ``Thousands of Problems for Theorem Provers,'' an online database of problems \cite{tptp}. The relevant input in this case is shown in Figure~\ref{figure:tptp}. The syntax takes some getting used to. Variables are capitalized, whereas predicates and function symbols have to begin with lower case letters. A universal quantifier $\forall X$ is written {\tt ![X]}. The designation {\tt fof} stands for ``first-order formula,'' the designations {\tt ax4} to {\tt ax12} and {\tt conc} are arbitrary labels, the word {\tt axiom} indicates the starting assumptions, and the word {\tt conjecture} indicates the desired conclusion. Incidentally, the TPTP library contains dozens of Smullyan puzzles encoded in this way. We do not know of any such puzzle in which the hypotheses were found to be inconsistent.

\begin{figure}[ht]
{\small
\begin{verbatim}
  fof(ax4, axiom, ![X] : peculiar(X) <=> (sane(X) <=> ~doctor(X))).
  fof(ax5, axiom, ![X] : special(X) <=>
    ![Y] : ~doctor(Y) <=> (sane(Y) <=> peculiar(X))).
  fof(ax7, axiom, ![X] : ![Y] : (sane(X) <=> special(Y)) =>
    (sane(bf(X)) <=> ~doctor(Y))).
  fof(ax8, axiom, sane(tarr) <=> ![X] : doctor(X) => sane(X)).
  fof(ax10, axiom, sane(fether) <=> ![X] : ~doctor(X) => ~sane(X)).
  fof(ax12, axiom, sane(fether) <=> sane(tarr)).
  fof(conc, conjecture, false).
\end{verbatim}
}
\caption{TPTP input to Vampire.}
\label{figure:tptp}
\end{figure}

More generally, one can ask Vampire to try to show that a particular conclusion (other than false) follows from a set of hypotheses. Saying that a conclusion $C$ follows from some hypotheses $H$ is equivalent to saying that $H$ together with the negation of $C$ is contradictory. One therefore asks Vampire to try to refute the combination of $H$ and the negation of $C$. If $C$ doesn't follow from $H$, the search can, in principle, go on forever. In practice, Vampire eventually times out with a message like ``Termination reason: Refutation not found.'' In some circumstances, Vampire can determine that the implication doesn't hold, in which case it reports ``Termination reason: Satisfiable.'' That means that there is a model of $H$ in which $C$ is false.

In the case at hand, Vampire provided us with proof of a contradiction from the hypotheses, a sequence of about two hundred lines, each of which followed from previous lines and the assumptions using a specific rule of inference. The proof left something to be desired, however. For one thing, theorem provers generally put formulas in a format that is convenient for proof search but not for human consumption, and the rules of inference have the same character. An even bigger problem is the overabundance of detail. Ordinary mathematical proofs are designed to highlight the key ideas and organize information in order to convey intuition as to \emph{why} a theorem holds. It is still an important research problem to develop ways of extracting such presentations from formal search tools. Fortunately, encouraged by Vampire's success, we were able to find more readable proofs on our own. These are presented in the next section.

\section{Ordinary proofs}
\label{section:ordinary:proofs}

In this section, we show that the hypotheses of Smullyan's last asylum are inconsistent. We present informal arguments, which can be viewed either as sketches of first-order derivations or as proofs that the hypotheses cannot be jointly satisfied in any model.

We have to be careful when reasoning about Smullyan's asylums. The operant notion of ``belief'' is counterintuitive, since an inhabitant can have inconsistent beliefs. For example, given an asylum in which an inhabitant, Jones, is a doctor and there is at least one patient, an insane inhabitant will believe, simultaneously, that every inhabitant is a doctor and that Jones is not a doctor. So even when we know that Smith believes that every inhabitant is a doctor, we have to resist the temptation to conclude that Smith believes that Jones is a doctor. In the arguments below, we therefore take great pains to justify each inference with the explicit rules that Smullyan has given us.

Smullyan himself presented a proof that the hypotheses imply that all the doctors in the asylum are insane and all the patients are sane. Our first proof picks up from there.

\begin{theorem}
The hypotheses enumerated in Section~\ref{section:the:puzzle} are inconsistent.
\end{theorem}

\begin{proof}
Suppose the hypotheses hold. By Smullyan's argument, all the doctors in the asylum are insane and all the patients are sane. This means that both the doctors and the patients believe that they are patients. By definition, this means that everyone is peculiar.

We next show that every inhabitant is special. Let $x$ be any inhabitant. Since everyone is peculiar, $x$ is peculiar. Since all the patients are sane, they believe that $x$ is peculiar, and since all the doctors are insane, no doctor believes that $x$ is peculiar. By definition, this means that $x$ is special.

We now use condition C to derive a contradiction. By hypothesis 6, there is a sane patient, because all the doctors are insane. Let $x$ be a sane patient, and let $y$ be $x$'s best friend. If $y$ is sane, then $y$ believes that Fether is a doctor. If $y$ is insane, then $y$ believes that $x$ is a doctor. Either way, there is someone, $z$, that $y$ believes to be a doctor.

Since everyone is special, $z$ is special, and since $x$ is sane, $x$ believes that $z$ is special. By condition C, $y$ believes that $z$ is a patient. But we have chosen $z$ so that $y$ believes that $z$ is a doctor, a contradiction.
\end{proof}

We now provide an alternative proof that is more direct and uses only six of the hypotheses.

\begin{theorem}
Hypotheses 4, 5, 7, 8, 10, and 12 are inconsistent.
\end{theorem}

\begin{proof}
From 8, 10, and 12, there are two possibilities: either Tarr is sane, in which case Fether is sane, all doctors are sane, and all patients are insane; or Tarr is insane, in which case Fether is insane, there is at least one insane doctor, and there is at least one sane patient. We consider each case in turn.

First, assume Tarr is sane, in which case all doctors are sane and all patients are insane. Since Tarr is sane, Tarr is a doctor.

By 4, nobody is peculiar, because being peculiar amounts to being a sane patient or an insane doctor. By 5, everyone is special, since for each inhabitant $x$, all the doctors are sane and believe that $x$ is not peculiar, and all the patients are insane and believe that $x$ is peculiar.

Since Tarr is sane, Tarr believes that he himself is special. By 7 (condition C), Tarr's best friend believes that Tarr is a patient. So Tarr's best friend is insane, and hence a patient. But Tarr also believes that his best friend is special, and hence, by 7, Tarr's best friend believes that they are a patient. But Tarr's best friend \emph{is} a patient, which contradicts the fact that Tarr's best friend is insane.

So we can assume that Tarr is insane, which means that there is at least one insane doctor and there is at least one sane patient. Let $x$ be an insane doctor and let $y$ be a sane patient. Then $x$ and $y$ are both peculiar, by 4. We now consider two cases.

First, suppose all the doctors are insane and all the patients are sane. Then no doctor believes that $x$ is peculiar and every patient believes that $x$ is peculiar, and so, by 5, $x$ is special. The same argument shows that $y$ is also special. Since $y$ is sane, $y$ believes that $x$ is special. By 7, $y$'s best friend believes that $x$ is a patient. Since $x$ is a doctor, $y$'s best friend is insane. By the same token, since $y$ is sane, $y$ believes that $y$ is special, and so, by 7, $y$'s best friend believes that $y$ is a patient. Since $y$ \emph{is} a patient, $y$'s best friend must be sane, a contradiction.

So either there is a sane doctor or an insane patient. Suppose $z$ meets either description. If $z$ is a sane doctor, then $z$ believes that $x$ is peculiar, and so, by 5, $x$ is not special. If $z$ is an insane patient, then $z$ believes that $x$ is not peculiar, and so again by 5, $x$ is not special. So, in either case, $x$ is not special. The same argument with $y$ in place of $x$ establishes that $y$ is not special.

Because $x$ is insane, $x$ believes that $x$ is special. By 7, $x$'s best friend believes that $x$ is a patient. Since $x$ is a doctor, $x$'s best friend is insane. Because $x$ is insane, $x$ also believes that $y$ is special. By 7, $x$'s best friend believes that $y$ is a patient. Since $y$ \emph{is} a patient, $x$'s best friend must be sane, a contradiction.
\end{proof}

\paragraph{Acnowledgements.}
We are grateful to Vedran \v{C}a\v{c}i\'{c} and two anonymous referees for helpful comments and suggestions.
Bentkamp's research has received funding from a Chinese Academy of Sciences President's International Fellowship for Postdoctoral Researchers (grant No.\ 2021PT0015).
This work has been partially supported by the Hoskinson Center for Formal Mathematics at Carnegie Mellon.


\begin{thebibliography}{99}

\bibitem{avigad} Avigad, J. (2018). The mechanization of mathematics. \emph{Not. Am. Math. Soc.} 65(6): 681--690.

\bibitem{avigad:book} Avigad, J. (to appear). \emph{Mathematical Logic and Computation}. Cambridge University Press.

\bibitem{lamr} Avigad, J., Heule, M.~J.~H., Nawrocki, W. \emph{Logic and Mechanized Reasoning}. \url{https://avigad.github.io/lamr/}.

\bibitem{enderton} Enderton, H.~B. (2001). \emph{A Mathematical Introduction to Logic}, second edition. San Diego: Harcourt.

\bibitem{handbook:satisfiability} Biere, A., Heule, M.~J.~H., van Maaren, H., Walsh, T. (2021). \emph{Handbook of Satisfiability}, second edition. Amsterdam: IOS Press.

\bibitem{harrison} Harrison, J. (2009). \emph{Handbook of Practical Logic and Automated Reasoning}. New York: Cambridge University Press.

\bibitem{heule:kullmann} Heule, M.~J.~H.. Kullman, O. (2017). The science of brute force. \emph{Commun. ACM}. 60(8): 70--79.

\bibitem{hofstadter} Hofstadter, D.~R. (1981). Metamagical themas. \textit{Scientific American}. 252(1). Reprinted as \cite{hofstadter2}.

\bibitem{hofstadter2} Hofstadter, D.~R. (1985). On self-referential sentences. In: Hofstadter, D.~R. {Metamagical Themas}. New York: Basic Books, pp.~5--24.

\bibitem{kovacs:voronkov} Kov\'acs, L., Voronkov, A. (2013).
First-order theorem proving and Vampire. In: Sharygina, N., Veith, H. eds. \textit{Computer Aided Verification (CAV) 2013}. Berlin: Springer, pp.~1--35.

\bibitem{kroening:strichman} 	Kroening, D., Strichman, O. (2008). \emph{Decision Procedures: An Algorithmic Point of View}, second edition. Berlin: Springer.

\bibitem{quine} Quine, W.~V.~O.~(1962). Paradox. \emph{Scientific American}. 206(4). Reprinted as \cite{quine2}.

\bibitem{quine2} Quine, W.~V.~O.~(1966). The ways of paradox. In: Quine, W.~V.~O. \emph{The Ways of Paradox and Other Essays}. Cambridge: Harvard University Press, pp.~1--21.

\bibitem{handbook:automated:reasoning} Robinson, J.~A., Voronkov, A. (2001) eds. \emph{Handbook of Automated Reasoning}, in two volumes. Amsterdam: Elsevier and MIT Press.

\bibitem{rautenberg} Rautenberg, W. (2010). \emph{A Concise Introduction to Mathematical Logic}, third edition. New York: Springer.

\bibitem{smullyan} Smullyan, R. (1982). The asylum of Doctor Tarr and Professor Fether. \emph{Coll. Math. J.}. 13(2): 142--146. Reprinted with minor modifications in \cite{smullyan3}.

\bibitem{smullyan2} Smullyan, R. (1982). The asylum of Doctor Tarr and Professor Fether: solutions. \emph{Coll. Math. J.}. 13(3): 213--217. Reprinted with minor modifications in \cite{smullyan3}.

\bibitem{smullyan3} Smullyan, R. (1982). \emph{The Lady or the Tiger? and Other Logic Puzzles}. New York: Alfred A.~Knopf, pp.~29--44.

\bibitem{tptp} Sutcliffe, G., Suttner, S. The TPTP problem library for automated theorem proving. \url{http://www.tptp.org/}.

\end{thebibliography}
\end{document}